\documentclass{birkjour}
\usepackage[utf8]{inputenc}
\usepackage[T1]{fontenc}
\usepackage[english]{babel}
\DeclareMathOperator{\diam}{diam}

\DeclareMathOperator{\Lip}{Lip}
\newcommand{\ind}{{\bf 1}}
\newcommand{\Nset}{\mathbb{N}}
\newcommand{\Rset}{\mathbb{R}}
\newcommand{\Prob}{\mathbb{P}}
\newcommand{\Cplusa}{C^+_\beta}
\newcommand{\PMalpha}{\mathcal{N}_{M,\alpha}}
\def\Nma{\PMalpha}
\newcommand{\FF}{\mathcal{F}}
\newcommand{\BB}{\mathcal{B}}
\newcommand{\MM}{\mathcal{M}}
\newcommand{\FM}[1]{\|#1\|_{\text{\rm FM}}}

\newcommand{\wstawka}[1]{\;\;\text{#1}\;\;}

\newcommand{\muprim}{\mu}
\newcommand{\deltaSp}{\delta_{(\Gamma,p)}}%

\theoremstyle{plain}% default
\newtheorem{theorem}{Theorem}%[section]
\newtheorem{lemma}[theorem]{Lemma}

\theoremstyle{definition}
\newtheorem{defn}{Definition}%[section]
\newtheorem{remark}[theorem]{Remark}%[section]
\begin{document}

\author{Wojciech Czernous}
\address[Wojciech Czernous]{University of Gdańsk, 
%Faculty of Oceanography and Geography, Institute of Oceanography, Department of Physical Oceanography, Laboratory of Marine Dynamics,
Al. Marszałka Piłsudskiego 46, 81-378 Gdynia, Poland}
\email[Corresponding author]{wojciech.czernous@ug.edu.pl}

% second author
\author{Tomasz Szarek}
\address[Tomasz Szarek]{Institute of Mathematics, University of Gdańsk, Wita Stwosza 57, 80-952 Gdańsk, Poland}
\email{szarek@intertele.pl}
\thanks{Tomasz Szarek was supported by the Polish
NCN grant 2016/21/B/ST1/00033}

\title[Generic homeomorphic IFSs]{Generic invariant measures for iterated systems of interval homeomorphisms}
\date{}

\begin{abstract}
It is known that Iterated Function Systems generated by orientation preserving homeomorphisms of the unit interval admit a unique invariant measure on $(0,1)$. The setup for this result is the positivity of Lyapunov exponents at both fixed points and the minimality of the induced action. With the additional requirement of continuous differentiability of maps on a fixed neighborhood of $\{0,1\}$, we present a metric in the space of such systems, which renders it complete. Using then a classical argument (and an alternative uniqueness proof), we show that almost singular invariant measures are admitted by systems lying densely in the space. This allows us to construct a residual set of systems with unique singular stationary distribution. Dichotomy between singular and absolutely continuous unique measures, is assured by taking a subspace of systems with absolutely continuous maps; the closure of this subspace is where the residual set is found.
\end{abstract}

%60J25 
% Continuous-time Markov processes on general state spaces

\dedicatory{To the memory of Professor Józef Myjak}

\maketitle

\section{Introduction}

Ergodic properties of random dynamical systems have been extensively studied for many years (see \cite{Arnold1998} and the references therein). The most important concepts related to these properties are  attractors and invariant measures (see \cite{Barnsley, CF, MD, S}). Among random systems, Iterated Function Systems and related  skew products are studied (see \cite{AM, AM1, BM}).
The note is concerned with Iterated Function Systems generated by orientation preserving homeomorphisms on the interval $[0, 1]$. It was proved (see \cite{Homburg2016}) that under quite general conditions the system
has a unique invariant measure  on the open interval $(0, 1)$. (Clearly, such systems have also trivial invariant measures supported on the endpoints.) 
It is well known that this measure is either absolutely continuous or singular with respect to Lebesgue measure.The methods that would allow us to distinguish each of these cases are still unknown. The main purpose of our paper is to support the conjecture that typically it should be singular (see \cite{AM}). In fact we prove that the set of Iterated Functions Systems that have a unique invariant measure on the interval $(0, 1)$ which is singular is generic in the natural class of systems.

\section{Notation}
By $\MM_1$ and $\MM_{fin}$ we denote the set of all probability measures and all finite measures
on the $\sigma$-algebra of Borel sets $\mathcal{B}([0,1])$, respectively. By $C([0,1])$ we denote the family of
all continuous functions equipped with the supremum norm $\|\cdot\|$. 

An operator $P:\MM_{fin}\to\MM_{fin}$ is called a {\it Markov operator} if it satisfies the following
two conditions:
\begin{enumerate}
	\item[1)] positive linearity: $P(\lambda_1\mu_1+\lambda_2\mu_2)=\lambda_1 P\mu_1+\lambda_2 P\mu_2$
	for $\lambda_1$, $\lambda_2\ge0$; $\mu_1$, $\mu_2\in\MM_{fin}$;
	\item [2)] preservation of  measure: $P\mu([0,1])=\mu([0,1])$ for $\mu\in\MM_{fin}$.
\end{enumerate}
A Markov operator $P$ is called a {\it Feller operator} if there is a linear operator $P^*:C([0,1])\to C([0,1])$
(dual to $P$) such that
\[
\int_{[0,1]} P^* f(x)\mu(dx) = \int_{[0,1]} f(x) P\mu(dx)
\wstawka{for}
f\in C([0,1]),
\mu \in\MM_{fin}.
\]
Note that, if such an operator exists, then $P^*(\ind)=\ind$. Moreover, $P^*(f)\ge 0$ if $f\ge 0$ and
\[
\| P^*(f) \|
\le \| P^*(|f|) \|
\le \| f \|,
\]
so $P^*$ is a continuous operator. A measure $\mu_*$ is called {\it invariant} if $P\mu_*=\mu_*$.

Let $\Gamma=\{g_1,\ldots,g_k\}$ be a finite collection of nondecreasing absolutely continuous functions,
mapping $[0,1]$ onto $[0,1]$, and let $p=(p_1,\ldots,p_k)$ be a probabilistic vector.
Clearly, it defines a probability distribution $p$ on $\Gamma$,
by putting $p(g_i)=p_i$.
Put $\Sigma_n=\{1,\ldots,k\}^n$, and let $\Sigma_*=\bigcup_{n=1}^\infty\Sigma_n$ 
be the collection of all finite words with entries from $\{1,\ldots,k\}$.
For a sequence $\omega\in\Sigma_*$, $\omega=(i_1,\ldots,i_n)$
we denote by $|\omega|$ its length (equal to $n$). 
Finally, let $\Omega=\{1,\ldots,k\}^{\Nset}$.
For $\omega=(i_1,i_2,\ldots)\in\Omega$ and $n\in\Nset$ 
we set $\omega|_n=(i_1,i_2,\ldots,i_n)$.
Let $\mathbb{P}$ be the product measure distribution on
$\Omega$ generated by the initial distribution on $\{1,\ldots,k\}$.
To any $\omega=(i_1,\ldots,i_n)\in\Sigma_{*}$, there corresponds a composition
$
g_\omega
=g_{i_n,i_{n-1},\ldots,i_1}
=g_{i_n}\circ g_{i_{n-1}}\circ\cdots\circ g_{i_1}
$.
The pair $(\Gamma,p)$, called an {\it Iterated Function System},
generates a Markov operator $P:\MM_{fin}\to\MM_{fin}$ of the form
\[
P\mu = \sum_{i=1}^{k} p_i\, g_i\mu,
\]
where $g_i\mu(A)=\mu(g_i^{-1}(A))$ for $A\in\BB([0,1])$; which describes the evolution of distribution due to 
action of randomly chosen maps from the collection $\Gamma$. It is a Feller operator,
that is, its dual operator $P^*$, given by the formula 
\[
P^* f(x) = \sum_{i=1}^{k} p_i f( g_i(x))
\quad \wstawka{for} 
f\in C([0,1]),
\;
x\in[0,1],
\]
has the property $P^*(C[0,1]) \subset C([0,1])$.

\begin{defn}\label{def1}
	We denote by $\Cplusa$, $\beta>0$, the space of continuous functions $g:[0,1]\to[0,1]$ satisfying the following properties:
	\begin{enumerate}
		\item[(1)] $g$ is nondecreasing, $g(0)=0$ and $g(1)=1$,
		\item[(2)] $g$ is continuously differentiable 
		on $[0,\beta]$ and $[1-\beta,1]$.
	\end{enumerate}
\end{defn}
We introduce the space $\FF_0$ of pairs $(\Gamma,p)$ such that 
$\Gamma=\{g_1,\ldots,g_k\}\subset\Cplusa$ is a finite collection of homeomorphisms,
and $p=(p_1,\ldots,p_k)$ is a probabilistic vector.
We endow this space with the metric
\[
d((\Gamma,p),(\Delta,q)) = \sum_{i=1}^{k}
\Big(
|p_i-q_i| 
+ \|g_i-h_i\|
+ \|g_i^{-1}-h_i^{-1}\|
+ \sup_{[0,\beta]\cup[1-\beta,1]}|g_i'-h_i'|
\Big),
\]
where $\Delta=\{h_1,\ldots,h_k\}$ and $q=(q_1,\ldots,q_k)$.
It is easy to check that $(\FF_0,d)$ is a complete metric space.

\begin{defn}[Admissible semigroups]\label{def2}
	Let $\Gamma=\{g_1,\ldots,g_k\}\subset\Cplusa$ be a finite collection of homeomorphisms,
	and let $p=(p_1,\ldots,p_k)$ be a probabilistic vector such that $p_i>0$ for all $i=1,\ldots,k$.
	The pair $(\Gamma,p)$ is called an {\it admissible Iterated Function System} if
	\begin{enumerate}
		\item[(1)] for any $x\in(0,1)$ the exist $i,j\in\{1,\ldots,k\}$ such that $g_i(x)<x<g_j(x)$;
		\item[(2)] the Lyapunov exponents at both common fixed points $0$, $1$ are positive, i.e.,
		\[
		\sum_{i=1}^{k}p_i\log g_i'(0) > 0
		\wstawka{and}
		\sum_{i=1}^{k}p_i\log g_i'(1) > 0.
		\]
	\end{enumerate}
\end{defn}

We consider the space $\FF\subset\FF_0$ of admissible Iterated Function Systems
$(\Gamma,p)$ with all maps $g\in\Gamma$ absolutely continuous.
The closure $\overline{\FF}$ of $\FF$ in $(\FF_0,d)$ is clearly complete.
Recall that a subset of a complete metric space is called \emph{residual} if its 
complement is a set of the first category.

Let $\MM_1(0,1)$ denote the space of $\mu\in\MM_1$
supported in the open interval $(0,1)$, i.e., satisfying $\mu((0,1))=1$.
Clearly, each $(\Gamma,p)\in\FF$ has two invariant measures: $\nu_1=\delta_0$ and $\nu_2=\delta_1$.
As we will show, it admits also a unique invariant measure in $\MM_1(0,1)$.

Since $P\mu$ is singular for singular $\mu\in\MM_{fin}$,
a unique invariant measure is either singular or absolutely continuous (see \cite{LasotaMyjak1994}).
The aim of this paper is to show that the set of all $(\Gamma,p)\in\overline{\FF}$, which have singular unique invariant measure
$\mu\in\MM_1(0,1)$, is residual in $\overline{\FF}$.

\section{Invariant measure}

The paper \cite{Homburg2016} deals with a special case of an IFS, namely for $k=2$,
with $g_1(x)<x$, $g_2(x)>x$, on $(0,1)$, being twice continuously differentiable.
The proof of \cite[Lemma~3.2]{Homburg2016}, gives the following result.
\begin{lemma}\label{lem01bis}
	Let $(\Gamma,p)$ be an admissible Iterated Function System.
	Then there exists an ergodic invariant measure $\mu\in\MM_1(0,1)$.
\end{lemma}

	The main ingredient of the proof of Lemma \ref{lem01bis} is the following subset of $\MM_1(0,1)$:
	for small $0<\alpha<1$ and positive $M$, one defines
	\begin{equation}\label{eq:Nc}
	\PMalpha = 
	\{
		\mu\in\MM_1(0,1) : \mu([0,x]) \le Mx^\alpha\;\text{and}\; \mu( [1-x,1] ) \le Mx^\alpha
	\}.
	\end{equation}
	For each admissible IFS, thanks to the positivity of Lyapunov exponents,
	the corresponding Markov operator $P$ keeps an $\PMalpha$ invariant,
	for some $M$ and $\alpha$.
	Then, among the invariant measures for $P$, given by the Krylov-Bogolyubov procedure,
	one can find an ergodic one, due to the convexity of $\PMalpha$.
	
	For the later use, we give below an easy extension of the invariance part 
	of the  \cite[Lemma~3.2]{Homburg2016} and its proof.
\begin{lemma}\label{lem01}
	Let $(\Gamma,p)$ be an admissible Iterated Function System and let $P$ 
	be the Markov operator corresponding to $(\Gamma,p)$. There exists $M>0$ and $\alpha>0$
	such that $\PMalpha\neq\emptyset$ and 
	\[
	P(\PMalpha) \subset \PMalpha.
	\]
\end{lemma}
\begin{proof}
	Note that, by Definitions \ref{def1} and \ref{def2},
	we can find $x_0>0$ and $\lambda_1,\ldots,\lambda_k>0$ such that 
	\[
	\Lambda :=\sum_{i=1}^{k} p_i\log \lambda_i > 0
	\]
	and 
	\[
	g_i^{-1}(x) \le x/\lambda_i \quad\wstawka{for all} x\in[0,x_0].
	\]
	Let $F(\alpha) = \sum_i p_i e^{-\alpha \log\lambda_i}$.
	Writing the Taylor expansion at $0$ we obtain 
	\begin{multline*}
	F(\alpha) =\sum_i p_i (1-\alpha\log \lambda_i+\mathcal{O}(\alpha^2))
	=  1-\alpha\sum_i p_i \log \lambda_i +\mathcal{O}(\alpha^2)
	\\
	=  1-\alpha\Lambda +\mathcal{O}(\alpha^2) < 1
	\end{multline*}
	for $\alpha$ small enough; choose such $\alpha>0$. Choose also $M>0$ so large that 
	\[
	M x_0^\alpha \ge 1.
	\]
	Assume that 
	$x\in[0,x_0]$ 
	and let $\mu\in\PMalpha$. Then
	\begin{equation*}
	\begin{split}
	P\mu([0,x]) 
	& 
	=     \sum_i p_i\mu([0,g_i^{-1}(x)]) 
	\le   \sum_i p_i\mu([0,x/\lambda_i]) 
	\le M \sum_i p_i x^\alpha/\lambda_i^\alpha \\
	&
	=   Mx^\alpha \sum_i p_i e^{-\alpha\log\lambda_i}
	=   Mx^\alpha \cdot F(\alpha) 
	\le Mx^\alpha.
	\end{split}
	\end{equation*}
	On the other hand, if $x>x_0$, then 
	$
	P\mu([0,x]) \le 1 \le M x_0^\alpha \le M x^\alpha.
	$
	In the same way we prove that $P\mu([1-x,1])\le Mx^\alpha$ 
	(with possibly larger $M$ and smaller $\alpha$).
	This completes the proof.
\end{proof}

Using ideas from \cite[Lemma~3]{Ilyashenko2010}, one can prove the following
\begin{lemma}\label{fullsupp}
	Let $(\Gamma,p)$ be an admissible Iterated Function System.
	Assume that there exist $g_1,g_2\in\Gamma$ such that 
	the ratio $\ln g_1'(0)/\ln g_2'(0)$ is irrational
	and 
	$g_1'(0)>1>g_2'(0)$;
	let $g_1'$ be H\"older continuous in a neighborhood of zero.
	Then the IFS $(\Gamma,p)$ is minimal, that is, 
	$\{g_{i_n,i_{n-1},\ldots,i_1}(x)\}_n$ is dense in $[0,1]$  for each $x$ and $\omega$.
	\qed
\end{lemma}

Uniqueness may be shown, under conditions of the preceding Lemma,
for instance using an argument given in \cite[Lemma~3.4]{Homburg2016};
we give another proof, which does not rely on injectivity of each $g\in\Gamma$.
\begin{lemma}\label{lem:exuniq}
	Under conditions of the preceding Lemma, there is a unique stationary measure in $\MM_1(0,1)$.
\end{lemma}
\begin{proof}
Let $M>0$ be such that $P(\Nma)\subset \Nma$
and let $\nu\in\Nma$.
Put
	\[
		\mu_n = \frac{ \nu+P\nu+\ldots+P^{n-1}\nu }{n}
	\]
	and note that, by \eqref{eq:Nc}, $\mu_n\in \Nma$.
	Let $\mu$ be an accumulation point of the sequence $(\mu_n)$ in the $*$-weak topology in $C([0,1])^*$.
	Then it is easy to check that also $\mu\in\Nma$.
	Moreover, since $P$ is a Feller operator, every accumulation point of the sequence $\mu_n$ is an invariant measure 
	for the process $P$.

We now prove the uniqueness. 
	Let $\mu\in\mathcal{M}_1(0,1)$ be an arbitrary invariant measure.
	Fix $f\in C([0,1])$.
	We define a sequence of random variables $(\xi_n^{f,\mu})_{n\in\Nset}$
	by the formula
	\[
		\xi_n^{f,\mu}(\omega) = \int_{[0,1]} f( g_{i_1,\ldots,i_n} (x) ) \mu(dx)
		\quad \;\text{for}\; \omega=(i_1,i_2,\ldots).
	\]
Since $\mu$ is an invariant measure for $P$, we easily check that $(\xi_n^{f,\mu})_{n\in\Nset}$
is a bounded martingale with respect to the natural filtration. Note that this martingale depends on the measure $\mu$.
From the Martingale Convergence Theorem it follows that $(\xi_n^{f,\mu})_{n\in\Nset}$ is convergent $\Prob$-a.s.
	and since the space $C([0,1])$ is separable, there exists a subset $\Omega_0$ of $\Omega$ with $\Prob(\Omega_0)=1$
	such that $(\xi_n^{f,\mu})_{n\in\Nset}$ is convergent for any $f\in C([0,1])$ and $\omega\in\Omega_0$.
	Therefore for any $\omega\in\Omega_0$ there exists a measure $\omega(\mu)\in\mathcal{M}_1$ such that 
	\[
		\lim_{n\to\infty} \xi_n^{f,\mu}(\omega) = \int_{[0,1]} f(x) \omega(\mu)(dx)
		\quad \;\text{for every}\; 
		f\in C([0,1]).
	\]
	We are now ready to show that for any $\varepsilon>0$ there exists $\Omega_{\varepsilon}\subset\Omega$
	with $\Prob(\Omega_{\varepsilon})=1$ satisfying the following property:
	for every $\omega\in\Omega_{\varepsilon}$ there exists an interval $I$ of length $|I|\le\varepsilon$ 
	such that $\omega(\mu)(I)\ge 1-\varepsilon$.
	Hence we obtain that $\omega(\mu) = \delta_{v(\omega)}$ 
	for all $\omega$ from some set $\tilde\Omega$ with $\Prob(\tilde\Omega)=1$.
	Here $v(\omega)$ is a point from $[0,1]$.

	Fix $\varepsilon>0$ and let $a,b\in(0,1)$ be such that $\mu([a,b]) > 1-\varepsilon$.
	Let $\ell\in\Nset$ be such that $1/\ell < \varepsilon/2$.
	Since for any $x\in(0,1)$ there exists $i\in\{1,\ldots,k\}$ such that $g_i(x)<x$,
	we may find a sequence $({\bf j}_n)_{n\in\Nset}$, ${\bf j}_n\in\Sigma_*$,
	such that $g_{{\bf j}_n}(b) \to 0$ as $n\to\infty$.
	Therefore, there exist ${\bf i}_1,\ldots,{\bf i}_\ell$ such that
	\[
		g_{{\bf i}_m}([a,b]) \cap g_{{\bf i}_n}([a,b]) = \emptyset
		\quad\;\text{
		for 
		}\;
		m,n\in\{1,\ldots,\ell\},
		m\neq n.
	\]
	Put $n^*=\max_{m\le \ell} |{\bf i}_m|$
	and set 
	$J_m = g_{{\bf i}_m}([a,b])$ for $m\in\{1,\ldots,\ell\}$.
	Each $J_m$ is a closed interval.
	(If we skipped the reqirement of injectivity of $g_i$'s, $J_m$ would possibly be a singleton; which does not spoil the proof; see Remark 4 afterwards).
	Now observe that for any sequence ${\bf u} = (u_1,\ldots,u_n)\in\Sigma_*$ 
	there exists $m\in\{1,\ldots,\ell\}$ such that $\diam(g_{{\bf u}}(J_m)) < 1/\ell < \varepsilon/2$.
	This shows that for any cylinder in $\Omega$, defined by fixing the first initial $n$ entries $(u_1,\ldots,u_n)$,
	the conditional probability, that $(u_1,\ldots,u_n,\ldots,u_{n+k})$
	are such that $\diam(g_{u_1,\ldots,u_n,\ldots,u_{n+k}}([a,b])) \ge \varepsilon/2$
	for all $k=1,\ldots,n^*$,
	is less than $1-q$ for some $q>0$ independent of $n$.
	Hence there exists $\Omega_\varepsilon \subset \Omega$ with $\Prob(\Omega_\varepsilon) = 1$
	such that for all $(u_1,u_2,\ldots) \in \Omega_{\varepsilon}$ we have $\diam(g_{u_1,\ldots,u_n}([a,b])) < \varepsilon/2$
	for infinitely many $n$.
	Since $[0,1]$ is compact, we may additionally assume that for infinitely many $n$'s
	the set $g_{u_1,\ldots,u_n}([a,b])$ is contained in a set $I$ (depending on $\omega=(u_1,u_2,\ldots)$)
	with $\diam(I)\le\varepsilon$.
	Observe that this $I$ does not depend on $\mu$;
	if only $\mu_1$, $\mu_2$ are two stationary distributions,
	and if $a$, $b$ are chosen in such a way that 
	$\mu_1([a,b]) > 1-\varepsilon$ 
	and
	$\mu_2([a,b]) > 1-\varepsilon$ 
	then we obtain
	that both 
	$\omega(\mu_1)(I)\ge 1-\varepsilon$
	and
	$\omega(\mu_2)(I)\ge 1-\varepsilon$,
	$\Prob$-a.s.
	Hence $\omega(\mu_1)=\omega(\mu_2)=\delta_{v(\omega)}$
	for $\Prob$-almost every $\omega\in\Omega$.
	Consequently, for any $f\in C([0,1])$ we have, for $i=1,2$, 
	\begin{equation*}
	\int_{[0,1]} f(x) \mu_i(dx) 
	= \lim_{n\to\infty} \int_{[0,1]} f(x) P^n\mu_i(dx)
	= \int_\Omega \lim_{n\to\infty} \xi_n^{f,\mu_i}(\omega) \Prob(d\omega).
	\end{equation*}
	Since the last integral equals $\int_\Omega f(v(\omega)) \Prob(d\omega)$
	in both cases, and since
	$f\in C([0,1])$ was arbitrary, we obtain $\mu_1=\mu_2$ and the proof is complete.
\end{proof}

\begin{remark}\label{remNonHom}
	The proofs of Lemmata \ref{lem01}, \ref{fullsupp} and \ref{lem:exuniq}
	do not use the fact that $g_i$, restricted to $(\beta,1-\beta)$, is injective.
	Precisely, these results 
	hold even if we replace the word "homeomorphisms" 
	by "functions" in Definition \ref{def2}.
\end{remark}

\section{Auxiliary results and main theorem}
Recall that the $*$-weak topology in $C([0,1])^*$ is induced by the Fortet-Mourier norm 
\[
\FM{\nu} = \sup
\left\{
\langle f, \nu \rangle : f\in BL_1
\right\},
\]
where $BL_1$ is the space of all functions $f:[0,1]\to R$ 
such that $|f(x)|\le 1$ and $|f(x)-f(y)|\le|x-y|$ for $x,y\in [0,1]$.

\bigskip
Denote by $d_0$ the distance between arbitrary (not necessarily admissible) Iterated Function Systems, 
given by
\[
d_0((\{S_1,\ldots,S_k\},p),(\{T_1,\ldots,T_k\},q)) = \sum_{i=1}^{k} |p_i-q_i|  + \|S_i-T_i\|.
\]
\begin{lemma}\label{L1}
	Let $(S,p)$ and $(T,q)$ be given Iterated Function Systems and let $P_S$ and $P_T$ be the corresponding 
	Markov operators. Then for every $\mu_1$, $\mu_2\in\MM_{fin}$ we have
	\begin{equation}\label{eq32}
	\FM{(P_S-P_T)(\mu_1-\mu_2)} 
	\le
	(\mu_1([0,1])+\mu_2([0,1])) \:
	d_0( (S,p), (T,q) ).
	\end{equation}
\end{lemma}
A proof of this lemma may be found in \cite{Szarek1998}.
\begin{lemma}\label{lem03}
	Suppose that the Iterated Function Systems: $(S,p)$, and $(S_1,p)$, $(S_2,p)$, \ldots, are such that, for some $M>0$ and $\alpha\in(0,1)$,
	\[
		\mu_1,\mu_2,\ldots\in\PMalpha
	\;\wstawka{and} 
	\lim_{n\to\infty}d_0((S_n,p),(S,p))=0
	\]
	where $\mu_n$ is a stationary distribution for $(S_n,p)$.
	Then $(\mu_n)_n$ admits a subsequence weakly convergent to 
	a stationary distribution $\mu$ for $(S,p)$;
	moreover, $\mu\in\PMalpha$.
\end{lemma}
\begin{proof}
	By the Prohorov theorem, there is a subsequence $\mu_{n_k}$, converging in $(\MM_1,\FM{\cdot})$
	to some $\muprim\in\MM_1$. 
	We denote it by $\mu_n$ again, for convenience, so $\FM{\mu_n-\muprim}\to0$.
	Applying Lemma \ref{L1},
	we get
	\[
	\FM{P\mu_n - \mu_n} =
	\FM{(P-P_n)\mu_n} \le \mu_n([0,1])\;d_0((S_n,p),(S,p)) \to 0,
	\]	
	where $P$, $P_n$ are the Markov operators corresponding to $(S,p)$, $(S_n,p)$, respectively.
	The weak continuity of $P$ implies $P\mu_n\to P\muprim$ in $\FM{\cdot}$,
	so that $\FM{P\muprim-\muprim}=0$. 
	Now, $\muprim$ is an invariant measure for $P$, but we still have to check
	$\muprim\in\PMalpha$.
	Indeed, let $x\in(0,1)$ be arbitrary. Then
	\[
	\mu_n([x,1-x]) \ge 1-2Mx^\alpha \quad\wstawka{for all} n\in\Nset.
	\]
	By the Alexandrov Theorem, $\muprim([x,1-x])\ge\limsup_n \mu_n([x,1-x])\ge 1-2Mx^\alpha$,
	which means that $\muprim\in\PMalpha$.
\end{proof}
Let $\MM_1^\varepsilon$ be the set of those $\mu\in\MM_1$, for which 
there is a Borel set $A$ with the Lebesgue measure $m(A)<\varepsilon$ and such that
$\mu(A)>1-\varepsilon/2$.
\begin{remark}\label{rem02}
	Note that, if $\mu\in\MM_1$ is singular, then for every $\varepsilon>0$ there is $\delta>0$ such that 
	for $\nu\in\MM_1$, $\FM{\mu-\nu}<\delta$
	we have $\nu\in\MM_1^\varepsilon$.
	For a proof, see \cite{Diss}, Lemma 2.4.1.
\end{remark}
For every $n\in\Nset$, let $\FF_n\subset\FF$ be the set of all $(\Gamma,p)\in\FF$
with invariant measure $\mu\in\MM_1(0,1)\cap \MM_1^{1/n}$.

\begin{lemma}\label{lemFndense}
	For every $n\in\Nset$ the set $\FF_n$ is dense in the space $\FF$ endowed with the metric $d$.
\end{lemma}
\begin{proof}
	Fix $(\Gamma,p)\in\FF$, $n\in\Nset$ and $\varepsilon>0$. 
	We may assume that $(\Gamma,p)$ satisfies the requirements of the Lemma \ref{fullsupp}.
	It is easy to construct absolutely continuous modifications of $g_1$ (on $(\beta,1-\beta)$) in such a way,
	that we obtain a sequence $(\Gamma_m,p)_{m\in\Nset}\subset\FF$ satisfying 
	$d((\Gamma_m,p),(\Gamma,p))<\varepsilon$ and
	\[
	\lim_{m\to\infty}d_0((\Gamma_m,p),(S_0,p)) = 0
	\]
	where $S_0=\{g_0,g_2,\ldots,g_k\}$ and $g_0(x)=x_0$ for $x\in[u,v]\subset(\beta,1-\beta)$, provided $v-u<2\varepsilon$.
	As we have noted in the Remark \ref{remNonHom}, 
    due to the inclusion $[u,v]\subset(\beta,1-\beta)$,
the system $(S_0,p)$ has a unique stationary distribution $\mu_0\in\MM_1(0,1)$, and $\mu_0$ has full support in $[0,1]$.
	Since $\mu_0$ is invariant, we obtain $\mu_0(\{x_0\})>0$,
	so $\mu_0$ is not absolutely continuous.
	As remarked earlier, the uniqueness of $\mu_0$
	implies that it is either absolutely continuous or singular.
	
	Now, $\mu_0$ is singular and we would like to apply the 
	Remark \ref{rem02},
	to find $(\Gamma_m,p)\in\FF_n$ for some $m\in\Nset$.
	In the light of the
	Lemma \ref{lem03},
	it suffices to show that all $(\Gamma_m,p)$ have stationary distributions
	belonging to the same set $\PMalpha$, for some $M,\alpha>0$.
	But this follows easily from the proof of Lemma \ref{lem01},
	by taking account of the fact that
	the corresponding elements of $\Gamma_m$ are identical on $(0,\beta)\cup(1-\beta,1)$.
\end{proof}

\begin{theorem}
	The set of all $(\Gamma,p)\in\overline{\FF}$ which have a unique singular stationary distribution in $\MM_1(0,1)$
	is residual in $\overline{\FF}$.
\end{theorem}
\begin{proof}
	For $(\Gamma,p)\in\FF_n$, we find $\delta_0>0$ such that
	\[
		|x - \min_{1\le i\le k}g_i(x)| > \delta_0
		\wstawka{and}
		|x - \max_{1\le i\le k}g_i(x)| > \delta_0
		\quad\wstawka{on}
		(1/n,1-1/n).
	\]
	Moreover, to $(\Gamma,p)$ we adjoin a compact set $Z_{(\Gamma,p)}\subset[0,1]$ 
	such that 
	\[
	\mu_{(\Gamma,p)}(Z_{(\Gamma,p)}) > 1-\frac{1}{2n}
	\wstawka{and}
	m(Z_{(\Gamma,p)}) < \frac{1}{n},
	\]
	where $\mu_{(\Gamma,p)}\in\MM_1(0,1)$ is the invariant measure for ${(\Gamma,p)}$.
	Further, due to regularity of the Lebesgue measure,
	we can find a positive number $r_{(\Gamma,p)}$ such that
	\[
		m( O( Z_{(\Gamma,p)}, r_{(\Gamma,p)} ) ) < \frac1n,
	\]
	where $O(Z_{(\Gamma,p)},r_{(\Gamma,p)})$ is the open neighborhood of $Z_{(\Gamma,p)}$
	in $[0,1]$
	with the radius $r_{(\Gamma,p)}$.
	Denote by $A_{(\Gamma,p)}$ the set $[0,1]\setminus O( Z_{(\Gamma,p)}, r_{(\Gamma,p)} )$ 
	and consider the classical Tietze function $f_{(\Gamma,p)}:[0,1]\to\Rset$
	for the sets $Z_{(\Gamma,p)}$ and $A_{(\Gamma,p)}$ 
	given by the formula
	\[
	f_{(\Gamma,p)}(x) = \frac
	{\|x,A_{(\Gamma,p)}\|}
	{\|x,A_{(\Gamma,p)}\|+\|x,Z_{(\Gamma,p)}\|},
	\]
	where $\|x,A\|$ stands for the distance of point $x$ from the set $A$.
	We have $f_{(\Gamma,p)}=0$ for $x\in A_{(\Gamma,p)}$
	and $f_{(\Gamma,p)}(x)=1$ for $x\in Z_{(\Gamma,p)}$.
	Obviously, $|f_{(\Gamma,p)}|\le 1$ and $f_{(\Gamma,p)}$ is Lipschitz,
	with the Lipschitz constant $l_{(\Gamma,p)}>1$.
	
	By Lemma \ref{lem03}, and by the proof of Lemma \ref{lem01},
	the map taking $(\Gamma,p)\in\FF$ into its invariant measure $\mu_{(\Gamma,p)}\in\MM_1(0,1)$
	is continuous with respect to $(\FF,d)$ and $(\MM_1(0,1),\FM{\cdot})$.
	Hence, there is $\delta_1>0$
	such that for $(\Delta,q)\in\FF$ satisfying $d( (\Delta,q), (\Gamma,p) )<\delta_1$,
	we have 
	\begin{equation}\label{eqContinuity}
	\FM{\mu_{(\Gamma,p)} - \mu_{(\Delta,q)}} < \frac{1}{2n\cdot l_{(\Gamma,p)}},
	\end{equation}
	where $\mu_{(\Delta,q)}\in\MM_1(0,1)$ 
	is the invariant measure for $(\Delta,q)$.
	
	Additionally, since $(\Gamma,p)$ is admissible, 
	from the Definition \ref{def2} follows that 
	there is $\delta_2>0$ with the property
		\[
		\sum_{i=1}^{k}p_i\log (g_i'(0)-\delta_2) > 0
		\wstawka{and}
		\sum_{i=1}^{k}p_i\log (g_i'(1)-\delta_2) > 0.
		\]
	Now, for $(\Gamma,p)\in\FF_n$, with the aid of  $\delta_{(\Gamma,p)}=\min\{\delta_0/2,\delta_1,\delta_2\}$, 
	we define \[
	\widehat\FF = \bigcap_{n=1}^\infty \bigcup_{(\Gamma,p)\in\FF_n} B_{\overline{\FF}}((\Gamma,p),\deltaSp),
	\]
	where $B_{\overline{\FF}}((\Gamma,p),\deltaSp)$
	is an open ball in $(\overline{\FF},d)$ with center at $(\Gamma,p)$ and radius $\deltaSp$.
	Clearly $\widehat{\FF}$ as an intersection of open dense subsets of $\overline{\FF}$ is residual.

	We will show that if $(T,q)\in\widehat{\FF}$
	then $(T,q)$ has the unique singular stationary distribution supported on $(0,1)$.
	Let $(T,q)$ be fixed; it is a limit, in $(\FF,d)$, of the sequence 
	$((\Gamma,p)_n)_{n\in\Nset}$.
	As we have noted, $T=\{h_1,\ldots,h_k\}$ is 
	a collection of homeomorphisms, due to convergence in the metric $d$.
	By the same token,
	the functions $h_i$ are even of class $C^1$ on $[0,\beta]$ and on $[1-\beta,1]$.
	It is easy to check that $(T,q)$ satisfies the condition (1) from Definition \ref{def2},
	due to the proper choice of $\deltaSp<\delta_0$.
	Moreover, by the inequality $\deltaSp\le\delta_2$, the Lyapunov exponents 
	for $(T,q)$ are kept positive.
	This means that $(T,q)$ is an admissible Iterated Function System.
	Hence it admits a unique invariant measure $\mu_{(T,q)}\in\MM_1(0,1)$.
	Again by the Lemma \ref{lem03}, there is a subsequence $(\Gamma,p)_{n_k}$,
	with invariant measures converging to $\mu_{(T,q)}$ in $\FM{\cdot}$.
	For short,
	let us denote elements of this subsequence again by $(\Gamma,p)_n$,
	the corresponding invariant measures by $\mu_n$,
	and let $Z_n=Z_{(\Gamma,p)_n}$, 
	$O(Z_n,r_n)=O(Z_{(\Gamma,p)_n},r_{(\Gamma,p)_n})$
	for
	$n\in\Nset$.
	
	The obvious inequality $\Lip(l_n^{-1}f_n)\le 1$,
	where $f_n=f_{(\Gamma,p)_n}$ and $l_n=l_{(\Gamma,p)_n}$,
	implies that
	\[
	\left|
	\int_{[0,1]}f_n(x)\mu_n(dx)
	-
	\int_{[0,1]}f_n(x)\mu_{(T,q)}(dx)
	\right|
	\le l_n\FM{\mu_n-\mu_{(T,q)}}
	< \frac{1}{2n}.
	\]
	However, since $f_n(x)=1$ for $x\in Z_n$
	and $f_n(x)=0$ for $x\not\in O(Z_n,r_n)$,
	\[
	\mu_n(Z_n) - \mu_{(T,q)}(O( Z_n,r_n )) < \frac{1}{2n}.
	\]
	Thus 
	\[
		\mu_{(T,q)}( O(Z_n,r_n) )
		>
		\mu_n(Z_n)-\frac{1}{2n}
		>
		1
		-\frac{1}{2n}
		-\frac{1}{2n}
		=
		1
		-\frac{1}{n}.
	\]
	This and the inequality $m( O(Z_n,r_n) )<\frac1n$, $n\in\Nset$, prove that $\mu_{(T,q)}$
	is singular. The proof is completed.
\end{proof}

\end{document}